\documentclass[10pt, reqno]{amsart}
\usepackage{amsmath, amsthm, amscd, amsfonts, amssymb, graphicx, color}
\usepackage[bookmarksnumbered, colorlinks, plainpages]{hyperref}
\hypersetup{colorlinks=true,linkcolor=red, anchorcolor=green, citecolor=cyan, urlcolor=red, filecolor=magenta, pdftoolbar=true}
\usepackage{mathrsfs}

\newtheorem{theorem}{Theorem}[section]
\newtheorem{lemma}[theorem]{Lemma}

\newtheorem{corollary}[theorem]{Corollary}
\theoremstyle{definition}
\newtheorem{definition}[theorem]{Definition}
\newtheorem{example}[theorem]{Example}

\theoremstyle{remark}
\newtheorem{remark}[theorem]{Remark}
\numberwithin{equation}{section}

\begin{document}
\setcounter{page}{1}

\title[Geometry of division rings]{Geometry of division rings}

\author[Nikolaev]
{Igor V. Nikolaev$^1$}

\address{$^{1}$ Department of Mathematics and Computer Science, St.~John's University, 8000 Utopia Parkway,  
New York,  NY 11439, United States.}
\email{\textcolor[rgb]{0.00,0.00,0.84}{igor.v.nikolaev@gmail.com}}


\subjclass[2010]{Primary 12E15, 32J15; Secondary 46L85.}

\keywords{quaternion algebra, algebraic surface, Belyi's theorem.}


\begin{abstract}
We prove an analog of Belyi's  theorem for the algebraic
surfaces. 
Namely, any non-singular algebraic surface can be defined over 
a number field if and only it covers the complex projective plane 
with ramification at three knotted two-dimensional spheres. 

\end{abstract}

\maketitle

\section{Introduction}
Belyi's theorem says that any non-singular algebraic curve over a number 
field is a covering of the complex projective line $P^1(\mathbf{C})$ ramified at 
the three points $\{0,1,\infty\}$   [Belyi 1979] \cite[Theorem 4]{Bel1}. 
The aim of our note is an analog of Belyi's theorem for the algebraic 
surfaces based on an  approach of  \cite{Nik1}. 
Namely, we associate to the countable division ring  an   avatar,
see definition \ref{def1.1}.  
If the ring  is commutative (non-commutative, resp.),  then
its avatar is an algebraic curve (an algebraic surface, resp.) defined over the field 
$\mathbf{C}$.   
For example, an avatar of the ring of rationals   (rational quaternions, resp.) 
is the complex projective line  $P^1(\mathbf{C})$  (complex projective plane  $ P^2(\mathbf{C}) $, resp.) 
 Belyi's theorem follows from an embedding of rationals into the field of algebraic numbers
 \cite[Section 4]{Nik1}   and its analog for surfaces from an embedding of rational quaternions into a quaternion algebra,
 see Section 4 in below.  An extension of Belyi's theorem to  complex surfaces was studied independently in [Gonz\'alez-Diez 2008]
\cite{Gon1}.

 Recall that  an analogy between  the number fields 
  and  complex algebraic curves is well  known  [Eisenbud \& Harris 1999] \cite[p. 83]{EH}.
 The  Grothendieck's  theory of schemes  cannot explain  this relation
 [Manin 2006] \cite[Section 2.2]{Man1}. The problem can be solved in terms of the 
$C^*$-algebras  \cite{Nik1}. To explain the solution, let $R$ be a (discrete) associative ring,
let $M_2(R)$ be the matrix ring over $R$
and  let 
\begin{equation}
\rho: M_2(R)\to \mathscr{B}(\mathcal{H})
\end{equation}
be a self-adjoint representation of $M_2(R)$  by the bounded linear operators 
 on a Hilbert space $\mathcal{H}$.  Taking the norm-closure of $\rho(M_2(R))$ 
 in the strong operator topology,  one gets a $C^*$-algebra  $\mathscr{A}_R$. 
 Likewise, let $B(V, \mathcal{L}, \sigma)$ be the twisted 
homogeneous coordinate ring of a complex projective variety $V$, where $\mathcal{L}$ is an invertible sheaf and $\sigma$
is an automorphism of $V$ [Stafford \& van ~den ~Bergh 2001]  \cite[p. 173]{StaVdb1}. 
Recall that the  Serre $C^*$-algebra, $\mathscr{A}_V$,   is the norm closure of a
self-adjoint representation of the ring  $B(V, \mathcal{L}, \sigma)$
in $\mathscr{B}(\mathcal{H})$  \cite[Section 5.3.1] {N}. 
Finally, let $\mathscr{K}$ be the $C^*$-algebra of all compact operators on the Hilbert space 
 $\mathcal{H}$. 
 We refer the reader to \cite{Nik1} for the motivation and examples illustrating the following
 definition.  
\begin{definition}\label{def1.1}
 The complex projective variety $V$ is called an {\it avatar}
 of the ring $R$,   if there exists  a $C^*$-algebra homomorphism: 
 \begin{equation}\label{eq1.2}
  \mathscr{A}_V\to\mathscr{A}_R\otimes\mathscr{K}.
 \end{equation}
\end{definition}

\begin{theorem}\label{thm1.2}
{\bf (\cite{Nik1})}
Let $\overline{\mathbf{Q}}$ be the algebraic closure of the field  $\mathbf{Q}$.
Then:

\medskip
(i) $P^1(\mathbf{C})$ is an avatar of the field $\mathbf{Q}$;

\smallskip
(ii) the avatar of a field $K\subset \overline{\mathbf{Q}}$
is a non-singular algebraic curve $C(\overline{\mathbf{Q}})$;

\smallskip
(iii)  the field extension $\mathbf{Q}\subset K$ defines a covering
$C(\overline{\mathbf{Q}})\to P^1(\mathbf{C})$ ramified at the  points $\{0,1, \infty\}$. 
\end{theorem}
 
 \medskip
\begin{remark}
Belyi's theorem follows from item (iii) of theorem \ref{thm1.2}. Roughly speaking, 
it can be shown that each non-singular algebraic curve $C(\overline{\mathbf{Q}})$
is  the avatar of a field $K\subset \overline{\mathbf{Q}}$. 
We refer the reader to \cite[Section 4]{Nik1} for the details.  
\end{remark}

\medskip
To formalize our results, we use the following notation.
Denote by $\left({a,b\over F}\right)$ a quaternion algebra, i.e 
the algebra over a field $F$,    such that $\{1,i.j, ij\}$ is a basis for
the algebra and 
  \begin{equation}
  i^2=a, \quad j^2=b, \quad ji=-ij
  \end{equation}
   for some $a,b\in F^{\times}$ [Voight 2021] \cite[Section 2.2]{V}. 
   The quaternion algebra with $a=b=-1$ will be written as $\mathbb{H}(F)$. 
  The algebra  $\mathbb{H}(\mathbf{R})$ corresponds   to the Hamilton  quaternions and 
   the algebra  $\mathbb{H}(\mathbf{Q})$  corresponds to the rational quaternions.  
  Our main result is a generalization of theorem \ref{thm1.2} to the division rings $\left({a,b\over K}\right)$,
 where $K\subset\overline{\mathbf{Q}}$. 
\begin{theorem}\label{thm1.4}
Let $\left({a,b\over K}\right)$ be a division ring, such that $K\subset\overline{\mathbf{Q}}$. 
 Then:

\medskip
(i) $P^2(\mathbf{C}) $ is an avatar of the division ring $\mathbb{H}(\mathbf{Q})$;

\smallskip
(ii)  the avatar of a division ring  $\left({a,b\over K}\right)$ is a non-singular
algebraic surface  $S(\overline{\mathbf{Q}})$;

\smallskip
(iii)  the field extension $\mathbf{Q}\subset K$ defines a covering
$S(\overline{\mathbf{Q}})\to P^2(\mathbf{C})$ ramified at 
  three  knotted two-dimensional spheres
$P^1(\mathbf{C})\cup P^1(\mathbf{C})\cup P^1(\mathbf{C})$.
\end{theorem}

\medskip
\begin{remark}
A relation between algebraic surfaces and division rings follows from \cite[Section 7.2]{N}. 
Indeed, each non-singular algebraic surface is a smooth 4-dimensional manifold and
the arithmetic topology relates such manifolds to the cyclic division algebras
\cite[Theorem 7.2.1]{N}.  
\end{remark}
\begin{remark}
The knotting type of $P^1(\mathbf{C})\cup P^1(\mathbf{C})\cup P^1(\mathbf{C})$  depends  
on the arithmetic of the field $K$ 
and extends the Grothendieck's theory of {\it dessin d'enfant} to the case of algebraic surfaces.   
\end{remark}

\medskip
The paper is organized as follows.  A brief review of the preliminary facts is 
given in Section 2. Theorem \ref{thm1.4} is proved in Section 3. 
An analog of Belyi's theorem for the algebraic surfaces is  proved
in Section 4.

\section{Preliminaries}
This section is a brief review of the quaternion and Serre $C^*$-algebras.
We refer the reader to [Voight 2021] \cite[Section 2.2]{V} and \cite[Section 5.3.1]{N} 
for a detailed account. 
\subsection{Quaternion algebras}
\begin{definition}
The algebra  $\left({a,b\over F}\right)$ over a field $F$ is called a quaternion algebra
  if there exists $i,j\in \left({a,b\over F}\right)$ such that $\{1,i.j, ij\}$ is a basis for $\left({a,b\over F}\right)$ and 
  \begin{equation}
  i^2=a, \quad j^2=b, \quad ji=-ij
  \end{equation}
   for some $a,b\in F^{\times}$. 
\end{definition}
\begin{example}
If $F\cong\mathbf{R}$ and $a=b=-1$, then the quaternion algebra 
$\left({-1,-1\over \mathbf{R}}\right)$ consists  of the Hamilton quaternions $\mathbb{H}(\mathbf{R})$;
hence the notation.  
If $F\cong\mathbf{Q}$, then the quaternion algebra 
$\left({-1,-1\over \mathbf{Q}}\right)$ consists  of the rational quaternions $\mathbb{H}(\mathbf{Q})$. 
\end{example}

\medskip
A $\ast$-involution on $\left({a,b\over F}\right)$ is defined by the formula
$(1,i,j,k)\mapsto (1, -i,-j,-k)$.  The norm $N(u)=uu^*$ of an element $u=x_0+xi+yj+ zk\in\left({a,b\over F}\right)$ 
is a quadratic form: 
\begin{equation}\label{eq2.2} 
N(x_0+xi+yj+ zk)=x_0^2-ax^2-by^2+abz^2. 
\end{equation}

\medskip
Since $N(1)=1$ and $N(uv)=N(u)N(v)$, one concludes that the  $\left({a,b\over F}\right)$ is a division algebra if 
and only if   quadratic form (\ref{eq2.2}) vanishes only at the zero element $u=0$. 
Thus  (\ref{eq2.2}) must be a positive form 
for all $x_0, x, y, z\in F$. 

\medskip
It is easy to see, that the form   (\ref{eq2.2}) admits non-trivial zeroes if and only if there are 
such zeroes for the ternary quadratic form: 
\begin{equation}\label{eq2.3} 
\mathcal{Q}(x,y,z)=-ax^2-by^2+abz^2. 
\end{equation}

\medskip
The substitution $a'={1\over b}, ~b'={1\over a}$ maps the zeroes of (\ref{eq2.3})  to the $F$-points of 
a conic surface given by the equation: 
\begin{equation}\label{eq2.4} 
z^2=ax^2+by^2. 
\end{equation}

\medskip
The following classification of the quaternion algebras is well known. 
\begin{theorem}\label{thm2.3}
{\bf (\cite[Theorem 5.1.1]{V})}
The formula
\begin{equation}
\left({a,b\over F}\right)\mapsto \mathcal{Q}(x,y,z)
 \end{equation}
maps  isomorphic quaternion algebras to the similar ternary quadratic forms. 
Equivalently, the quaternion algebras are classified by the isomorphism classes
of the conic surfaces (\ref{eq2.4}). 
\end{theorem}

\subsection{Serre $C^*$-algebras}
Let $V$ be an $n$-dimensional complex  projective variety endowed with an automorphism $\sigma:V\to V$  
 and denote by $B(V, \mathcal{L}, \sigma)$   its  twisted homogeneous coordinate ring  [Stafford \& van ~den ~Bergh 2001]  \cite{StaVdb1}.
Let $R$ be a commutative  graded ring,  such that $V=Proj~(R)$.  Denote by $R[t,t^{-1}; \sigma]$
the ring of skew Laurent polynomials defined by the commutation relation
$b^{\sigma}t=tb$  for all $b\in R$, where $b^{\sigma}$ is the image of  $\left({a,b\over \mathbf{Q}}\right)$ under automorphism 
$\sigma$.  It is known, that $R[t,t^{-1}; \sigma]\cong B(V, \mathcal{L}, \sigma)$.

Let $\mathcal{H}$ be a Hilbert space and   $\mathscr{B}(\mathcal{H})$ the algebra of 
all  bounded linear  operators on  $\mathcal{H}$.
For a  ring of skew Laurent polynomials $R[t, t^{-1};  \sigma]$,  
 consider a homomorphism: 
\begin{equation}\label{eq2.1}
\rho: R[t, t^{-1};  \sigma]\longrightarrow \mathscr{B}(\mathcal{H}). 
\end{equation}
Recall  that  $\mathscr{B}(\mathcal{H})$ is endowed  with a $\ast$-involution;
the involution comes from the scalar product on the Hilbert space $\mathcal{H}$. 
We shall call representation (\ref{eq2.1})  $\ast$-coherent,   if
(i)  $\rho(t)$ and $\rho(t^{-1})$ are unitary operators,  such that
$\rho^*(t)=\rho(t^{-1})$ and 
(ii) for all $b\in R$ it holds $(\rho^*(b))^{\sigma(\rho)}=\rho^*(b^{\sigma})$, 
where $\sigma(\rho)$ is an automorphism of  $\rho(R)$  induced by $\sigma$. 
Whenever  $B=R[t, t^{-1};  \sigma]$  admits a $\ast$-coherent representation,
$\rho(B)$ is a $\ast$-algebra.  The norm closure of  $\rho(B)$  is   a   $C^*$-algebra
   denoted  by $\mathscr{A}_V$.  We  refer to  $\mathscr{A}_V$  as   the   Serre $C^*$-algebra
 of the complex projective variety $V$.

\section{Proof}
\subsection{Part I} 
Let us prove item (i) of theorem \ref{thm1.4}. 
Denote by $\mathcal{O}$ the ring of integers of the quaternion algebra  $\left({a,b\over \mathbf{Q}}\right)$. 
Consider a polynomial ring:
\begin{equation}\label{eq3.1}
\mathfrak{R}=\mathbf{Z}[x,y,z]/[\mathcal{Q}],
\end{equation}
where $[\mathcal{Q}]$ is an ideal generated by the quadratic form $\mathcal{Q}(x,y,z)$ given by formula 
(\ref{eq2.3}).  The proof of item (i) is based on the following lemma.

\begin{lemma}\label{lm3.1}
The  matrix rings $M_2(\mathcal{O})$ and  $M_2(\mathfrak{R})$  are isomorphic. 
\end{lemma}
\begin{proof}
Roughly speaking, lemma \ref{lm3.1}  follows from the classification of the quaternion algebras given by  Theorem \ref{thm2.3}. 
Namely, the quaternion algebras are classified by the conic surfaces defined by (\ref{eq2.4}) or, equivalently,  by their coordinate rings  
 $\mathfrak{R}$. The same is true for the corresponding matrix rings.  Let us pass do a detailed argument.

\begin{figure}[h]
\begin{picture}(300,110)(-70,0)
\put(20,70){\vector(0,-1){35}}
\put(130,70){\vector(0,-1){35}}
\put(45,23){\vector(1,0){60}}
\put(45,83){\vector(1,0){60}}
\put(0,20){$\mathcal{Q}(x,y,z)$}
\put(55,30){${\sf similarity}$}
\put(45,90){${\sf isomorphism}$}
\put(117,20){$\mathcal{Q}'(x,y,z)$}
\put(7,80){$\left({a,b\over \mathbf{Q}}\right)$}
\put(115,80){$\left({a',b'\over \mathbf{Q}}\right)$}
\put(0,50){$F$}
\put(140,50){$F$}
\end{picture}
\caption{
}
\end{figure}

\medskip
(i) One can recast theorem \ref{thm2.3} as a commutative diagram in Figure 1.
We wish  to upgrade the map $F$  to a ring isomorphism. 
The simplest non-commutative ring attached  to the ternary quadratic form $\mathcal{Q}(x,y,z)$
is the matrix ring  $M_2(\mathfrak{R})$ over the ring $\mathfrak{R}$  defined by (\ref{eq3.1}).   
On the other hand, the quaternion algebra  $\left({a,b\over \mathbf{Q}}\right)$ is simple and, therefore, cannot be isomorphic
to  $M_2(\mathfrak{R})$. However, the ring of integers $\mathcal{O}$ of the algebra $\left({a,b\over \mathbf{Q}}\right)$ admits non-trivial 
two-sided ideals.  Here again, the ring $\mathcal{O}$  cannot be isomorphic to  $M_2(\mathfrak{R})$,
since $\mathcal{O}$ is a domain while  $M_2(\mathfrak{R})$ admits the zero divisors, e.g. the  projections. 
Thus we must consider the matrix ring $M_2(\mathcal{O})$ as a candidate for the required 
ring isomorphism.  Let us show that  $M_2(\mathcal{O})\cong M_2(\mathfrak{R})$  whenever $ \mathcal{Q}(x,y,z)=F\left({a,b\over \mathbf{Q}}\right).$

\bigskip
(ii)  Indeed, it follows from  [Voight 2021] \cite[Corollary 5.5.2]{V}  that the similar  
quadratic forms $\mathcal{Q}(x,y,z)$ and $ \mathcal{Q}'(x,y,z)$ correspond to 
 the isomorphic   conic surfaces  (\ref{eq2.4}) and, therefore, to 
the isomorphic rings  $\mathfrak{R}$ and   $\mathfrak{R}'$. 
Since $\mathcal{O}\subset \left({a,b\over \mathbf{Q}}\right)$, we conclude that an isomorphism between  $\left({a,b\over \mathbf{Q}}\right)$ and 
$\left({a',b'\over \mathbf{Q}}\right)$
induces an isomorphism between  $\mathcal{O}$ and  $\mathcal{O}'$.  In other words,  the
diagram in Figure 2 must be  commutative.

\begin{figure}[h]
\begin{picture}(300,110)(-70,0)
\put(20,70){\vector(0,-1){35}}
\put(130,70){\vector(0,-1){35}}
\put(45,23){\vector(1,0){60}}
\put(45,83){\vector(1,0){60}}
\put(15,20){$\mathfrak{R}$}
\put(45,30){${\sf isomorphism}$}
\put(45,90){${\sf isomorphism}$}
\put(125,20){$\mathfrak{R}'$}
\put(17,80){$\mathcal{O}$}
\put(125,80){$\mathcal{O}'$}
\put(0,50){$F$}
\put(140,50){$F$}
\end{picture}
\caption{
}
\end{figure}

\bigskip
(iii) The tensor product  with the matrix ring $M_2$ in Figure 2 gives us
 $\mathcal{O}\otimes M_2\cong M_2(\mathcal{O})$ and  $\mathfrak{R}\otimes M_2\cong M_2(\mathfrak{R})$. 
The functor $F$ extends to the tensor product and  one gets  a 
commutative diagram in Figure 3.

\bigskip
(iv) It remains to show that the map $F$ defines a ring isomorphism
  $M_2(\mathcal{O})\cong M_2(\mathfrak{R})$. 
  Indeed,  let $x_0\in M_2(\mathcal{O})$. The left multiplication
  $y\mapsto x_0y$ (addition $y\mapsto x_0+y$, resp.) 
  defines a morphism $\phi_{x_0}^{mult}$ ($\phi_{x_0}^{add}$, resp.)
  of the ring   $M_2(\mathcal{O})$.  Since the map $F$ preserves morphisms,
  we conclude that $F(\phi_{x_0}^{mult})$  ($F(\phi_{x_0}^{add})$, resp.) 
  is a morphism of the ring  $M_2(\mathfrak{R})$. 
  Namely, the morphism $F(\phi_{x_0}^{mult})$   ($F(\phi_{x_0}^{add})$, resp.) 
  acts  by the formula $F(y)\mapsto F(x_0)F(y)$ 
  ($F(y)\mapsto F(x_0)+F(y)$, resp.) 
  Thus one gets $F(x_0y)=F(x_0)F(y)$ and  $F(x_0+y)=F(x_0)+F(y)$
  for all $x_0,y\in M_2(\mathcal{O})$. In other words, the map $F$ defines 
  an an isomorphism  between the rings 
  $M_2(\mathcal{O})$ and $M_2(\mathfrak{R})$.

 \bigskip 
  Lemma \ref{lm3.1} is proved. 
   \end{proof}

\begin{lemma}\label{lm3.2}
Conic surface  (\ref{eq2.4}) is an avatar of the quaternion algebra  $\left({a,b\over \mathbf{Q}}\right)$. 
\end{lemma}
\begin{proof}
(i) 
According to definition \ref{def1.1},
we  must consider a self-adjoint representation
$\rho$ of the rings $M_2(\mathfrak{R})\cong M_2(\mathcal{O})$
by the bounded linear operators
on a Hilbert space $\mathcal{H}$.  We take the norm-closure 
of $\rho$ in the strong operator topology.  Lemma \ref{lm3.1}
implies the following isomorphism of the $C^*$-algebras:
\begin{equation}\label{eq3.3}
\rho(M_2(\mathfrak{R}))\otimes\mathscr{K}\cong \rho(M_2(\mathcal{O}))\otimes\mathscr{K}. 
\end{equation}

\bigskip
(ii) On the other hand, from the definition of the Serre $C^*$-algebra $\mathcal{A}_V$ one gets
the following isomorphisms: 
\begin{equation}\label{eq3.4} 
\begin{cases} 
\rho(M_2(\mathfrak{R}))\cong \mathcal{A}_{V(\mathbf{Q})}&\cr
\rho(M_2(\mathfrak{R}))\otimes\mathscr{K}\cong \mathcal{A}_{V(\mathbf{C})}, &
\end{cases}
\end{equation}
 where  $V(\mathbf{Q})$ ($V(\mathbf{C})$, resp.) is the conic surface (\ref{eq2.4})
 over the field of rational numbers $\mathbf{Q}$ (complex numbers $\mathbf{C}$, resp.) 
Thus the LHS of (\ref{eq3.3}) is the Serre $C^*$-algebra of the  conic surface  (\ref{eq2.4}).

\bigskip
(iii) It is immediate that the $\rho(M_2(\mathcal{O}))$ at the RHS of  (\ref{eq3.3}) is the $C^*$-algebra $\mathscr{A}_R$ of 
the  ring  $R\cong \left({a,b\over \mathbf{Q}}\right)$.

\bigskip
(iv)  Using (ii) and (iii), we can write (\ref{eq3.3}) in the form: 
\begin{equation}\label{eq3.5}
\mathcal{A}_{V(\mathbf{C})} \cong \mathscr{A}_R\otimes\mathscr{K}. 
\end{equation}

\bigskip
(v) It remains to compare (\ref{eq3.5}) and  the definition \ref{def1.1},  
where the connecting map in (\ref{eq1.2}) is an isomorphism between the  $C^*$-algebras. 
We conclude that the conic surface  (\ref{eq2.4}) 
is an avatar of the quaternion algebra  $\left({a,b\over \mathbf{Q}}\right)$. 

\bigskip
Lemma \ref{lm3.2} is proved.
\end{proof}

\begin{figure}[h]
\begin{picture}(300,110)(-70,0)
\put(20,70){\vector(0,-1){35}}
\put(130,70){\vector(0,-1){35}}
\put(45,23){\vector(1,0){60}}
\put(45,83){\vector(1,0){60}}
\put(0,20){$M_2(\mathfrak{R})$}
\put(45,30){${\sf isomorphism}$}
\put(45,90){${\sf isomorphism}$}
\put(115,20){$M_2(\mathfrak{R}')$}
\put(0,80){$M_2(\mathcal{O})$}
\put(117,80){$M_2(\mathcal{O}')$}
\put(0,50){$F$}
\put(140,50){$F$}
\end{picture}
\caption{
}
\end{figure}

\begin{lemma}\label{lm3.3}
The  (\ref{eq2.4}) is a rational surface over the field $k\cong\mathbf{Q}(\sqrt{-1}, \sqrt{-a}, \sqrt{-b})$. 
In particular, the complex points of (\ref{eq2.4}) define a simply connected 4-dimensional manifold. 
\end{lemma}
\begin{proof}
(i) Let us show that the conic  (\ref{eq2.4}) is a rational surface over the field $\mathbf{Q}(\sqrt{-1}, \sqrt{-a}, \sqrt{-b})$. 
Indeed, the reader can verify that a parametrization  $(u,v)\mapsto (x,y, z)$  of  (\ref{eq2.4})
is given by the formulas:
\begin{equation}\label{eq3.6} 
\begin{cases} 
x={u^2-v^2\over\sqrt{-a}}&\cr
y={2uv\over\sqrt{-b}}&\cr
z=\sqrt{-1}(u^2+v^2) &
\end{cases}
\end{equation}
We conclude that  the conic  (\ref{eq2.4}) is a rational surface $P^2(k)$ 
defined over the field  $k\cong \mathbf{Q}(\sqrt{-1}, \sqrt{-a}, \sqrt{-b})$.

\bigskip
(ii) It is well known that the rational complex projective variety is  simply connected as a manifold. 
By item (i) surface (\ref{eq2.4}) is rational and therefore the underlying
4-dimensional manifold is simply connected. 

\bigskip
Lemma \ref{lm3.3} is proved.
\end{proof}

\begin{remark}
Notice that  $k\ne \mathbf{Q}$.  For otherwise 
the ternary quadratic form  (\ref{eq2.3}) admits (infinitely many) non-trivial  zeroes
and  $\left({a,b\over \mathbf{Q}}\right)$ is no longer a division
ring, see Section 2.1. 
 \end{remark}

\begin{corollary}\label{cor3.5}
 $P^2(\mathbf{C}) $ is an avatar of the division ring $\mathbb{H}(\mathbf{Q})$.
\end{corollary}
\begin{proof}
Our proof is based on the result  [Piergallini 1995] \cite{Pie1} which
says that for each smooth 4-dimensional manifold $M^4$
  there exists a transverse immersion $X\hookrightarrow S^4$ of a 2-dimensional surface $X$
  into the 4-dimensional sphere $S^4$,  such that $M^4$ is the 4-fold PL  cover of  $S^4$ branched at the points of $X$.  
  We pass to a detailed argument.

  \bigskip
  (i)  Recall that if  $J$ is the complex conjugation,  then  $P^2(\mathbf{C})/J\cong S^4$.  
  Using Piergallini's Theorem,  we conclude that $J$  acts on the conic  surface  (\ref{eq2.4})
  so that it becomes a branched cover of  $P^2(\mathbf{C})$. 
  In view of lemma \ref{lm3.3},  the  conic surface is rational over the field 
 $k\cong \mathbf{Q}(\sqrt{-1}, \sqrt{-a}, \sqrt{-b})$. 
Thus there exists a regular map:
\begin{equation}
P^2(k)\to \mathbb{P}^2(k_0),
\end{equation}
where $k_0\subset k$ is the minimal non-trivial subfield of $k$.

\bigskip
(ii) But the minimal non-trivial subfield of $k\cong \mathbf{Q}(\sqrt{-1}, \sqrt{-a}, \sqrt{-b})$
independent of the parameters $a$ and  $b$ coincides with the field  $k_0\cong \mathbf{Q}(\sqrt{-1})$. 
Clearly, the  $k_0$ corresponds to the case $a=b=-1$.
All other possible combinations $a=\pm 1, ~b=\pm 1$ must be  
excluded since the ternary quadratic form (\ref{eq2.3}) must be  positive-definite.

\bigskip
(iii) It remains to notice that the quaternion algebra with  $a=b=-1$ corresponds to the 
rational quaternions $\mathbb{H}(\mathbf{Q})$.

\bigskip
Corollary \ref{cor3.5} is proved. \end{proof}

\smallskip
Item (i) of theorem \ref{thm1.4} follows from corollary \ref{cor3.5}.

\bigskip
\subsection{Part II} 
Let us prove item (ii) of theorem \ref{thm1.4}.  We proceed with construction of  an algebraic 
surface $S(\overline{\mathbf{Q}})$ from the quaternion algebra  $\left({a,b\over K}\right)$
with $K\subset \overline{\mathbf{Q}}$. 
From Part I  if  $K\cong\mathbf{Q}$,    then  $S(\overline{\mathbf{Q}})$ is given 
by the equation $\mathcal{Q}(x,y,z)=0$,  where 
\begin{equation}\label{eq3.8}
\mathcal{Q}(x,y,z)=-ax^2-by^2+abz^2. 
\end{equation}

\bigskip
(i) Let $K\subset \overline{\mathbf{Q}}$ be a number field and let $a,b\in K$. Denote by $p\in\mathbf{Z}[u]$ and  $q\in\mathbf{Z}[w]$  the minimal 
polynomials of $a$ and  $b$, respectively.  We  set $a=h, ~b=w$ and  we write (\ref{eq3.8}) in the form:
\begin{equation}\label{eq3.9}
F(x,y,z, u, w)=-ux^2-wy^2+uwz^2\in \mathbf{Z}[x,y,z,u,w].
\end{equation}

\bigskip
(ii)  Solving the equation $F(x,y,z, u, w)=0$,   one gets:
\begin{equation}\label{eq3.10}
u={wy^2\over wz^2-x^2}, \quad w={ux^2\over uz^2-y^2}. 
\end{equation}

\bigskip
(iii) The required algebraic surface  $S(\overline{\mathbf{Q}})$ is defined  as an intersection of two  hyper-surfaces 
given  the equations: 
\begin{equation}\label{eq3.11} 
\begin{cases} 
p\left( {wy^2\over wz^2-x^2}  \right)=0,&\cr
q\left( {ux^2\over uz^2-y^2}\right)=0. &
\end{cases}
\end{equation}

\bigskip
(iv) The reader can verify that the surface (\ref{eq3.11}) is defined over the field $\overline{\mathbf{Q}}$
and coincides with the conic surface  (\ref{eq3.8}) when $K\cong \mathbf{Q}$.

\bigskip
Item (ii) of theorem \ref{thm1.4} is proved.

\bigskip
\subsection{Part III} Let us prove item (iii) of theorem \ref{thm1.4}.
For the sake of clarity, we  consider the case $K\cong \mathbf{Q}$ first, 
and then the general case $K\subset\overline{\mathbf{Q}}$. 

\bigskip
\subsubsection{Case $K\cong \mathbf{Q}$}
(i)  Lemma \ref{lm3.2} says that  conic surface  (\ref{eq2.4}) is an avatar of the quaternion algebra  $\left({a,b\over \mathbf{Q}}\right)$. On the other hand, it is known that (\ref{eq2.4}) 
is a rational surface $P^2(k)$ over the number field $k\cong\mathbf{Q}(\sqrt{-1}, \sqrt{-a}, \sqrt{-b})$, see lemma \ref{lm3.3}.

\bigskip
(ii)  Let $k_1\cong\mathbf{Q}(\sqrt{-1}), ~k_2\cong\mathbf{Q}(\sqrt{-a})$ and
$k_3\cong\mathbf{Q}(\sqrt{-b})$. 
Consider  a regular map $P^2(k_i)\to P^2(\mathbf{Q})$
between the rational surfaces   $P^2(k_i)$ and  $P^2(\mathbf{Q})$
shown at the lower level in Figure 4.

\begin{figure}[h]
\begin{picture}(300,140)(-70,0)

\put(63,100){$P^2(k)$}

\put(75,90){\vector(0,-1){20}}
\put(55,90){\vector(-1,-1){20}}
\put(95,90){\vector(1,-1){20}}

\put(20,50){$P^2(k_1)$}
\put(60,50){$P^2(k_2)$}
\put(110,50){$P^2(k_3)$}

\put(75,40){\vector(0,-1){20}}
\put(35,40){\vector(1,-1){20}}
\put(115,40){\vector(-1,-1){20}}

\put(60,0){$P^2(\mathbf{Q})$}

\end{picture}
\caption{
}
\end{figure}

\bigskip
(iii) Using  [Piergallini 1995] \cite{Pie1} (see proof of corollary \ref{cor3.5}),  
we conclude that each $P^2(k_i)\to P^2(\mathbf{Q})$
is a covering map ramified over an embedded 2-dimensional surface $X$. 
To determine the genus of $X$, recall that the group of deck transformations 
of the covering $P^2(k_i)\to P^2(\mathbf{Q})$ is isomorphic to the Galois group 
$Gal~(k_i | \mathbf{Q})\cong \mathbf{Z}/2\mathbf{Z}$ of the field $k_i$. 
In particular,  the ramification set $X$ is fixed by the deck transformations
and, therefore, corresponds to the field $\mathbf{Q}$ fixed by the group  
$Gal~(k_i | \mathbf{Q})$. But the avatar of $\mathbf{Q}$ is a 
projective line $P^1(\mathbf{C})\cong X$,  see item (i) of theorem \ref{thm1.2}. 
We conclude that the surface $X$ has genus zero.

\bigskip
(iv)   Since $k\cong\mathbf{Q}(\sqrt{-1}, \sqrt{-a}, \sqrt{-b})$,
one gets for each $i=1,2, 3$  a regular map $P^2(k)\to P^2(k_i)$ as shown at the upper level of diagram in Figure 4. Composing these  maps  with the maps $P^2(k_i)\to P^2(\mathbf{Q})$, one concludes the algebraic surface $P^2(k)$
is a covering of  $ P^2(\mathbf{C}) $ ramified over three  knotted two-dimensional spheres
$P^1(\mathbf{C})\cup P^1(\mathbf{C})\cup P^1(\mathbf{C})$ embedded in $ P^2(\mathbf{C})$.  
\begin{remark}
The knotting type of $P^1(\mathbf{C})\cup P^1(\mathbf{C})\cup P^1(\mathbf{C})$  depends  on the arithmetic of the fields $k_i$ 
and extends the Grothendieck's theory of {\it dessin d'enfant} to the case of algebraic surfaces.   
\end{remark}

\bigskip
\subsubsection{Case  $K\subset\overline{\mathbf{Q}}$}
(i)  Let  $S(\overline{\mathbf{Q}})$ is an avatar of the quaternion algebra 
$\left({a,b\over K}\right)$ constructed in Part II.
It follows from formulas (\ref{eq3.9})-(\ref{eq3.11}) that there exists 
a regular map:
\begin{equation}\label{eq3.12}
f:  S(\overline{\mathbf{Q}})\to P^2(k).
\end{equation}

\bigskip
(ii)  Let $f^{-1}(P^1(\mathbf{C})\cup P^1(\mathbf{C})\cup P^1(\mathbf{C}))$
be the pre-image of the three knotted two-dimensional spheres embedded in $P^2(\mathbf{C})$
under the map (\ref{eq3.12}).  It is not hard to see, that such a pre-image 
consists again of  three  spheres $P^1(\mathbf{C})\cup P^1(\mathbf{C})\cup P^1(\mathbf{C})$ but knotted differently when compared  
to the case of the surface  $P^2(k)$. 

\bigskip
(iii)  Since our surface is an avatar of the quaternion algebra 
$\left({a,b\over K}\right)$, it  is defined over $\overline{\mathbf{Q}}$, see item (ii) 
of theorem \ref{thm1.4}.  This argument finishes the proof of item (iii) 
of theorem \ref{thm1.4}.

\bigskip
Theorem \ref{thm1.4} is proved.

\section{Belyi's theorem for algebraic surfaces}
The aim of this section is an analog of Belyi's theorem for algebraic surfaces. 
Our approach is geometric and  follows from  theorem \ref{thm1.4}. 
We refer the reader to  [Gonz\'alez-Diez 2008] \cite{Gon1}
for an analytic treatment to this problem. 
\begin{theorem}\label{thm4.1}
A non-singular algebraic surface is defined over a number field $K$
if and only if it is a covering  of the complex projective plane $P^2(\mathbf{C}) $
ramified at three knotted two-dimensional spheres
$P^1(\mathbf{C})\cup P^1(\mathbf{C})\cup P^1(\mathbf{C})$.
\end{theorem}
\begin{proof}
Our proof is based on item (iii) of  theorem \ref{thm1.4} and the following lemma.
\begin{lemma}\label{lm4.2}
For each algebraic surface $S(\overline{\mathbf{Q}})$
 there exists a quaternion algebra  $\left({a,b\over K}\right)$
 such that the avatar of $\left({a,b\over K}\right)$ is isomorphic to  $S(\overline{\mathbf{Q}})$. 
\end{lemma}
\begin{proof}
(i) In view of [Piergallini 1995] \cite{Pie1} (see proof of corollary \ref{cor3.5}), 
the algebraic surface  $S(\overline{\mathbf{Q}})$ is a covering of the 
projective plane $P^2(\mathbf{C}) $ ramified over a knotted 2-dimensional surface  $X\subset P^2(\mathbf{C})$. 
In particular, there exists a regular map $\phi: S(\overline{\mathbf{Q}})\to  P^2(\mathbf{C}) $. 

\bigskip
(ii) Recall that the  $ P^2(\mathbf{C}) $ is an avatar  of the division ring $\left({-1,-1\over \mathbf{Q}}\right)$,
see item (i) of theorem \ref{thm1.4}.  Moreover, the field  inclusion
$\mathbf{Q}\subset\overline{\mathbf{Q}}$ gives rise to a regular map 
$\psi_0: S_0(\overline{\mathbf{Q}})\to  P^2(\mathbf{C}) $, where   $S_0(\overline{\mathbf{Q}})$
is an avatar of the quaternion algebra $\left({-1, -1\over \overline{\mathbf{Q}}}\right)$.

\begin{figure}[h]
\begin{picture}(300,140)(-70,0)

\put(35,100){$S(\overline{\mathbf{Q}})$}
\put(90,100){$S_0(\overline{\mathbf{Q}})$}

\put(65,103){\vector(1,0){20}}
\put(50,90){\vector(0,-1){70}}
\put(105,90){\vector(0,-1){20}}

\put(100,90){\vector(-2,-3){40}}

\put(85,50){ $\left({-1, -1\over K}\right)$}

\put(105,40){\vector(0,-1){20}}

\put(30,0){$P^2(\mathbf{C})$}
\put(90,0){$\left({-1,-1\over \mathbf{Q}}\right)$}
\put(65,3){\vector(1,0){20}}

\put(35,50){$\phi$}
\put(70,110){$\psi$}
\put(70,70){$\psi_0$}
\end{picture}
\caption{
}
\end{figure}

\bigskip
(iii)  One gets a regular map $\psi:  S(\overline{\mathbf{Q}})\to S_0(\overline{\mathbf{Q}})$
by closing arrows of the commutative diagram in Figure 5. Notice that $\psi$ is a finite covering of 
the surface $S_0(\overline{\mathbf{Q}})$, since $\phi$ and $\psi_0$ are mappings of finite degree.

\bigskip
(iv) On the other hand, it follows from  equations (\ref{eq3.9})-(\ref{eq3.11}) 
that each finite covering   $S(\overline{\mathbf{Q}})$ of 
the avatar $S_0(\overline{\mathbf{Q}})$
of the algebra  $\left({-1, -1\over K}\right)$ must 
be  avatar of an algebra  $\left({a, b\over K}\right)$
for some $a,b\in K$. Thus  there exists
a quaternion algebra   $\left({a, b\over K} \right)$,
such that  algebraic surface $S(\overline{\mathbf{Q}})$ is the 
avatar of  $\left({a, b\over K}\right)$.

Lemma \ref{lm4.2} is proved. 
\end{proof}

\bigskip
Returning to the proof of theorem \ref{thm4.1},  one combines lemma \ref{lm4.2} with the conclusion of 
item (iii) of theorem \ref{thm1.4}.  This argument finishes the proof of theorem \ref{thm4.1}. 
\end{proof}

\bibliographystyle{amsplain}


\end{document}